\documentclass[10pt,a4paper]{article}
\usepackage[utf8]{inputenc}
\usepackage[english]{babel}
\usepackage{hyperref}
\usepackage{amsmath,amsthm,enumerate}
\usepackage{amssymb}
\usepackage{tikz}
\usepackage{subfigure}
\usepackage[hmargin=2.5cm,vmargin=3cm]{geometry}
\usepackage[color=green!30]{todonotes}

\usepackage{perpage} %the perpage package
\MakePerPage{footnote} %the perpage package command

\newtheorem{theorem}{Theorem}

\newtheorem{proposition}[theorem]{Proposition}
\newtheorem{lemma}[theorem]{Lemma}
\newtheorem{observation}[theorem]{Observation}
\newtheorem{corollary}[theorem]{Corollary}

\newtheorem{definition}[theorem]{Definition}

\newtheorem{claim}{Claim}[theorem]
\newtheorem{problem}{Problem}

\newcommand{\TC}{TC}
\newcommand{\cH}{\mathcal{H}}
\newcommand{\cE}{\mathcal{E}}
\newcommand{\cC}{\mathcal{C}}
\newcommand{\VC}{vc}
\newcommand{\dVC}{dvc}

\newcommand{\claimproof}{\noindent\emph{Proof of claim.} }
\newcommand{\smallqed}{{\tiny ($\Box$)}}

\begin{document}

\title{Bounding the order of a graph using its diameter\\and metric dimension\\[2mm] \large A study through tree decompositions and VC dimension}
\author{Laurent Beaudou\footnote{\noindent LIMOS - CNRS UMR 6158, Universit\'e Blaise Pascal, Clermont-Ferrand (France). E-mails: laurent.beaudou@univ-bpclermont.fr, florent.foucaud@gmail.com}
\and Peter Dankelmann\footnote{\noindent Department of Pure and Applied Mathematics, University of Johannesburg (South Africa). E-mails: \{pdankelmann,mahenning\}@uj.ac.za.}
\and Florent Foucaud\footnotemark[1]~\footnotemark[2]%{Part of this work was done while this author was a postdoctoral fellow at the University of Johannesburg.}
\and Michael A. Henning\footnotemark[2]
\and Arnaud Mary\footnote{\noindent Univ Lyon, Universit\'e Lyon 1, CNRS, LBBE - UMR 5558, F69622 (France). E-mail: arnaud.mary@univ-lyon1.fr}
\and Aline Parreau\footnote{\noindent Univ Lyon, Universit\'e Lyon 1, CNRS, LIRIS - UMR 5205, F69622 (France). E-mail: aline.parreau@univ-lyon1.fr}
}

\maketitle

\begin{abstract}
The metric dimension of a graph is the minimum size of a set of vertices such that each vertex is uniquely determined by the distances to the vertices of that set. Our aim is to upper-bound the order $n$ of a graph in terms of its diameter $d$ and metric dimension $k$. In general, the bound $n\leq d^k+k$ is known to hold. We prove a bound of the form $n=\mathcal{O}(kd^2)$ for trees and outerplanar graphs (for trees we determine the best possible bound and the corresponding extremal examples). More generally, for graphs having a tree decomposition of width $w$ and length $\ell$, we obtain a bound of the form $n=\mathcal{O}(kd^2(2\ell+1)^{3w+1})$. This implies in particular that $n=\mathcal{O}(kd^{\mathcal{O}(1)})$ for graphs of constant treewidth and $n=\mathcal{O}(f(k)d^2)$ for chordal graphs, where $f$ is a doubly-exponential function. Using the notion of distance-VC dimension (introduced in 2014 by Bousquet and Thomass\'e) as a tool, we prove the bounds $n\leq (dk+1)^{t-1}+1$ for $K_t$-minor-free graphs, and $n\leq (dk+1)^{d(3\cdot 2^{r}+2)}+1$ for graphs of rankwidth at most $r$.
\end{abstract}

\section{Introduction}

A \emph{resolving set} of a graph is a set of vertices that uniquely determines each vertex by means of the ordered set of distances to the vertices in the resolving set. The \emph{metric dimension} of the graph is the smallest size of a resolving set. These concepts, introduced independently by Slater~\cite{S75} (who called resolving sets \emph{locating sets}) and by Harary and Melter~\cite{HM76}, are widely studied since then, see for example the papers~\cite{BC11,BDJO15,CEJO00,HMMPSW10,KRR96,ST04}. More generally, they fit into the topic of \emph{identification} or \emph{separation} problems in discrete structures, such as separating systems, distinguishing sets and related concepts (for a few references, see~\cite{BS07,B72,CCCHL08,R61}). These concepts have many applications and connections to other areas. For example, the metric dimension can be applied to network discovery~\cite{BEEHHMR06,BBDGKP11}, robot navigation~\cite{KRR96}, coin-weighing problems~\cite{ST04}, $T$-joins~\cite{ST04}, the Mastermind game~\cite{C83}, or chemistry~\cite{CEJO00}.

The goal of this paper is to study the relation between the order, the diameter and the metric dimension of graphs, in particular for graphs belonging to specific graph classes.

\bigskip

\noindent\textbf{Important concepts and definitions.} All considered graphs are finite and simple. We will denote by $N[v]$, the \emph{closed neighbourhood} of vertex $v$, and by $N(v)$ its \emph{open neighbourhood} $N[v]\setminus\{v\}$. Let $d_G(u,v)$, or simply $d(u,v)$ if there is no ambiguity, denote the distance between two vertices $u$ and $v$ in graph $G$. Similarly, for two sets $X$ and $Y$ of vertices of $G$, $d_G(X,Y)$ denotes the shortest distance between a vertex of $X$ and a vertex of $Y$.

\begin{definition}
A set $R$ of vertices of a graph $G$ is a \emph{resolving set} if for each pair $u,v$ of distinct vertices, there is a vertex $x$ of $R$ with $d(x,u)\neq d(x,v)$. The smallest size of a resolving set of $G$ is the \emph{metric dimension} of $G$.
\end{definition}

A graph is said to be \emph{chordal} if it has no induced cycle of length at least~$4$. A graph is \emph{planar} if it has an embedding in the plane that induces no edge-crossing. It is outerplanar if it is planar and has an embedding in the plane where each vertex lies on the outer face. A \emph{minor} of a graph is a graph obtained by a succession of vertex- and edge-deletions and edge-contractions. We say that a graph $G$ is $H$-minor-free if $H$ is not a minor of $G$. By the Graph Minor Theorem~\cite{seymour_minor_20}, any minor-closed class of graphs (such as the classes of planar graphs, outerplanar graphs or graphs with treewidth at most~$w$) is defined by a finite set of forbidden minors.

\bigskip

\noindent\textbf{Previous work.} One can easily observe that in a graph $G$ of diameter $d$ and with metric dimension $k$ and $n$ vertices, we have the bound $n\leq d^k+k$~\cite{CEJO00,KRR96}. Indeed, given a resolving set $R$ of size $k$, every vertex outside of $R$ can be associated to a distinct vector of length $k$ and values ranging from $1$ to $d$. This trivial bound, however, is only tight for $d\leq 3$ or $k=1$~\cite{HMMPSW10}. Nevertheless, the more precise (and tight) bound $n\leq \left(\lfloor 2d/3\rfloor+1\right)^k+k\sum_{i=1}^{\lceil d/3\rceil}(2i-1)^{k-1}$ is given in~\cite{HMMPSW10}. It is natural to ask for which kind of graphs a bound of this form is tight. We therefore wish to study the following problem.

\begin{problem}\label{problem} Given a graph class $\cC$, determine the largest possible order of a graph in $\cC$ having metric dimension $k$ and diameter $d$.
\end{problem}

This problem was considered by the third and fifth author, together with Mertzios, Naserasr and Valicov~\cite{part1}. These authors studied interval graphs and permutation graphs, and proved bounds of the form $n=\mathcal{O}(dk^2)$. These bounds were shown to be best possible (up to constant factors). In the case of unit interval graphs, bipartite permutation graphs and cographs, it was proved in the same paper that $n=\mathcal{O}(dk)$.

Surprisingly, the above problem seems to have not been studied even for trees, despite the fact that the metric dimension of trees is well understood (see~\cite{CEJO00,KRR96,S75}). In this paper, we answer this question. We extend our result for trees in two ways. First, we give bounds involving the length and width of a tree decomposition of the graph. Second, we study graphs that have bounded \emph{distance-VC dimension}. (These notions will be defined in the corresponding sections of the paper.)

As further recent work related to this paper, we remark that the metric dimension of $t$-trees has recently been investigated in~\cite{BDJO15}, and the treelength of a graph has recently been used to design algorithms to compute the metric dimension~\cite{BFGR15}. Algorithms and complexity results regarding the computation of the metric dimension of graphs belonging to graph classes considered in the present paper, can be found in~\cite{DPSV12,ELW12j,part2}.

\bigskip

\noindent\textbf{Our results and structure of the paper.} In the first part of the paper, Section~\ref{sec:TW}, we study trees and generalize our method using the tool of tree decompositions. We start in Section~\ref{sec:trees} by an exact bound of the form $n=(\tfrac{1}{8}+o(1))d^2k$ for trees of order $n$, metric dimension~$k$ and diameter~$d$, and we characterize the trees reaching our bound. We then show in Section~\ref{sec:treedec} that a graph with a tree decomposition of width $w$ and length $\ell$ satisfies $n=\mathcal{O}(kd^2(2\ell+1)^{3w+1})$. %[To check]
This implies the bound $n=\mathcal{O}(k2^{3w}d^{3w+3})$ for graphs of treewidth at most $w$ and $n=\mathcal{O}(kd^23^{3\omega-2})$ for chordal graphs with maximum clique $\omega$.

The second part of the paper, Section~\ref{sec:VC}, is devoted to the use of the distance-VC dimension. We first show (using the notion of test covers), how the VC dimension of the ball hypergraph of a graph can be used to derive a general bound on the order using the diameter and the metric dimension. We then bound the dual distance-VC dimension of $K_t$-minor-free graphs and graphs of rankwidth at most~$r$, which implies the bounds $n\leq (dk+1)^{t-1}+1$ and $n\leq (dk+1)^{d(3\cdot 2^{r}+2)}+1$, respectively. In particular, this shows that for planar graphs, we have $n\leq (dk+1)^4$; this partially answers an open question from~\cite{part1}. We then use a completely different method in Section~\ref{sec:outerplanar} to prove that $n=\mathcal{O}(kd^2)$ for outerplanar graphs, which we show to be tight.

Finally, we conclude in Section~\ref{sec:conclu} with some open questions.

\section{Trees and graphs with specific tree decompositions}\label{sec:TW}

We first study Problem~\ref{problem} for graphs admitting specific types of tree decompositions. We start with trees, which form the class of nontrivial graphs that is the simplest (with respect to tree decompositions).

\subsection{Trees}\label{sec:trees}

We first give the constructions of some extremal trees. See Figure~\ref{fig:extr-trees} for illustrations.

For $r \in \mathbb{N}$ let $L_r$ be the rooted tree obtained from
a path $v_0,v_1,\ldots,v_r$ rooted at $v_0$ by attaching a path of
length $r-i$ to vertex $v_i$ for $i=1,2,\ldots,r-1$. Denote a path
of length $r$ rooted at one of its end vertices by $P_r^*$.
Let $k\in \mathbb{N}$ with $k\geq 2$.

For even $d\in \mathbb{N}$
we define the \emph{hairy spider}
$HS_{d,k}$ as the tree obtained from $k$ disjoint copies
of $L_{d/2}$ and a path $P_{d/2}^*$ by identifying their roots to a vertex $v$.
For odd $d \in \mathbb{N}$ with $d\geq 3$ and for $a\in \mathbb{N}$ with
$0\leq a \leq k$ we define
$HS_{d,k,a}$ as the tree obtained from $k$ disjoint copies of $L_{(d-1)/2}$,
two copies of $P_{(d-1)/2}^*$, and a path on two vertices, $u$ and $w$,
by identifying the roots of $a$ copies of $L_{(d-1)/2}$ and of one
copy of $P_{(d-1)/2}^*$ with $u$, and the roots of the
remaining $k-a$ copies of $L_{(d-1)/2}$ and
the root of the other copy of $P_{(d-1)/2}^*$ with $w$.

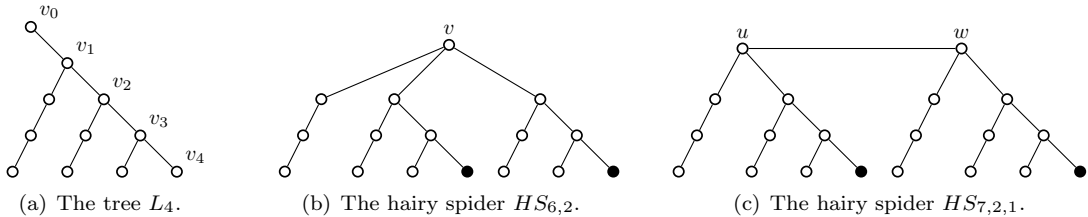
\begin{figure}[ht]
  \centering
  \subfigure[The tree $L_4$.]{
    \scalebox{0.8}{
      \begin{tikzpicture}[join=bevel,inner sep=0.6mm,line width=0.8pt, scale=0.3]

        \foreach \i in {0,...,4}{
          \path (2*\i,8-2*\i) node[draw, shape=circle] (\i-0) {};
          \draw (\i-0) node[above right=0.1cm] {$v_\i$};
          \ifnum \i>0
          \pgfmathtruncatemacro{\ii}{\i-1}
          \draw[line width=0.5pt] (\i-0)--(\ii-0);

          \ifnum \i<4
          \pgfmathtruncatemacro{\k}{4-\i}
          \foreach \j in {1,...,\k}{
            \path (2*\i-\j,8-2*\i-2*\j) node[draw, shape=circle] (\i-\j) {};
            \pgfmathtruncatemacro{\jj}{\j-1}
            \draw[line width=0.5pt] (\i-\j)--(\i-\jj);
          }
          \fi
          \fi
        }

  \end{tikzpicture}}}\qquad
  \subfigure[The hairy spider $HS_{6,2}$.]{
    \scalebox{0.8}{
      \begin{tikzpicture}[join=bevel,inner sep=0.6mm,line width=0.8pt, scale=0.3]

        \path (3,11) node[draw, shape=circle] (r) {};
        \draw (r) node[above=0.1cm] {$v$};

        \foreach \i in {0,...,2}{
          \path (-4-\i,8-2*\i) node[draw, shape=circle] (p-\i) {};
          \ifnum \i>0
          \pgfmathtruncatemacro{\ii}{\i-1}
          \draw[line width=0.5pt] (p-\i)--(p-\ii);
          \fi
        }

        \draw[line width=0.5pt] (p-0)--(r);

        \foreach \a in {0,1}{

          \foreach \i in {0,...,2}{

            \ifnum \i<2
            \path (\a*8+2*\i,8-2*\i) node[draw, shape=circle] (\a-\i-0) {};
            \fi

            \ifnum \i=2
            \path (\a*8+2*\i,8-2*\i) node[draw, shape=circle,fill] (\a-\i-0) {};
            \fi

            \pgfmathtruncatemacro{\ii}{\i-1}
            \ifnum \i>0
            \draw[line width=0.5pt] (\a-\i-0)--(\a-\ii-0);
            \fi

            \ifnum \i<2
            \pgfmathtruncatemacro{\k}{2-\i}
            \foreach \j in {1,...,\k}{
              \path (\a*8+2*\i-\j,8-2*\i-2*\j) node[draw, shape=circle] (\a-\i-\j) {};
              \pgfmathtruncatemacro{\jj}{\j-1}
              \draw[line width=0.5pt] (\a-\i-\j)--(\a-\i-\jj);
            }
            \fi
          }
          \draw[line width=0.5pt] (\a-0-0)--(r);
        }

  \end{tikzpicture}}}\qquad
  \subfigure[The hairy spider $HS_{7,2,1}$.]{\scalebox{0.8}{\begin{tikzpicture}[join=bevel,inner sep=0.6mm,line width=0.8pt, scale=0.3]

        \path (-2.5,10.8) node[draw, shape=circle] (r0) {};
        \draw (r0) node[above=0.1cm] {$u$};
        \path (9.5,10.8) node[draw, shape=circle] (r1) {};
        \draw (r1) node[above=0.1cm] {$w$};

        \draw[line width=0.5pt] (r0)--(r1);

        \foreach \a in {0,1}{

          \foreach \i in {0,...,2}{
            \path (\a*12+-4-\i,8-2*\i) node[draw, shape=circle] (\a-p-\i) {};
            \ifnum \i>0
            \pgfmathtruncatemacro{\ii}{\i-1}
            \draw[line width=0.5pt] (\a-p-\i)--(\a-p-\ii);
            \fi
          }

          \draw[line width=0.5pt] (\a-p-0)--(r\a);

          \foreach \i in {0,...,2}{

            \ifnum \i<2
            \path (\a*12+2*\i,8-2*\i) node[draw, shape=circle] (\a-\i-0) {};
            \fi

            \ifnum \i=2
            \path (\a*12+2*\i,8-2*\i) node[draw, shape=circle,fill] (\a-\i-0) {};
            \fi

            \pgfmathtruncatemacro{\ii}{\i-1}
            \ifnum \i>0
            \draw[line width=0.5pt] (\a-\i-0)--(\a-\ii-0);
            \fi

            \ifnum \i<2
            \pgfmathtruncatemacro{\k}{2-\i}
            \foreach \j in {1,...,\k}{
              \path (\a*12+2*\i-\j,8-2*\i-2*\j) node[draw, shape=circle] (\a-\i-\j) {};
              \pgfmathtruncatemacro{\jj}{\j-1}
              \draw[line width=0.5pt] (\a-\i-\j)--(\a-\i-\jj);
            }
            \fi
          }
          \draw[line width=0.5pt] (\a-0-0)--(r\a);
        }

  \end{tikzpicture}}}
  \caption{Extremal trees. Black vertices form optimal resolving sets.}
  \label{fig:extr-trees}
\end{figure}

\begin{theorem}\label{thm:trees}
Let $T$ be a tree of diameter $d$ and metric dimension $k$, where
$k\geq 2$. Then
\[ |V(T)| \leq \left\{ \begin{array}{cc}
     \frac{1}{8}(kd+4)(d+2) & \textrm{if $d$ is even,} \\
      \frac{1}{8}(kd-k+8)(d+1) & \textrm{if $d$ is odd}.
      \end{array} \right. \]
Equality holds for even $d$ if and only if $T=HS_{d,k}$, and for odd
$d$ if and only if $T=HS_{d,k,a}$ for some integer $a$ with $0<a<k$.
\end{theorem}

\begin{proof} Let $S=\{ x_1,x_2,\ldots,x_k\}$ be a
resolving set for $T$. Let $C$ be
the set of {\em central vertices} of $T$, that is, the set of vertices of $T$ that minimize the maximum distance to all the other vertices of the tree. For $i=1,2,\ldots,k$, let $P_i$
be the shortest path from $C$ to $x_i$.
Let $r = \max_{v\in V(T)} d_T(v,C)$, so if $d$ is even then $r=d/2$ and $r$
is the radius of $T$, and if $d$ is odd then $r=(d-1)/2$ since for odd
$d$ every vertex is within distance $(d-1)/2$ of the nearest central vertex
of $T$.
Define the subtree $T_S$ of $T$ by
\begin{equation} \label{eq:define-T_S}
T_S = T[C] \cup \bigcup_{i=1}^k P_i.
\end{equation}
For $v\in V(T_S)$ we define $T_v$ to be the largest subtree
of $T$ containing $v$ and no other vertex of $T_S$. In
other words, $T_v$ is the union of all branches of $T$ at $v$
not containing any edge of $T_S$. Possibly,  $T_v=K_1$. 
We first show that for every vertex $v$ of $T_S$,
\begin{equation}  \label{eq:every-branch-is-a-path}
\textrm{$T_v$ is a path with $v$ as an end-vertex.}
\end{equation}
We first show that $v$ is an end-vertex (that is, a leaf) of $T_v$.
Indeed, if $v$ had two neighbours in $T_v$, then they would have the same
distance to every vertex in $S$, and so $S$ would not resolve them,
a contradiction. The same argument shows that no vertex of
$T_v$ has degree greater than two. This shows \eqref{eq:every-branch-is-a-path}. \\[1mm]
Hence $T$ is obtained from $T_S$ by appending a path on $|V(T_v)|-1$
vertices to $v$ for all $v\in V(T_S)$.
We now bound the length of this path by showing that
\begin{equation} \label{eq:bound-on-T_v}
|V(T_v)| \leq r - d_T(v,C)+1.
\end{equation}
Let $v'$ be the end vertex of $T_v$ with $v'\neq v$.
Then $d_T(v',v) =|V(T_v)|-1$. Hence
\[ r    \geq d_T(v',C)
        =   d_T(v',v) +  d_T(v,C)
        =  |V(T_v)|-1 + d_T(v,C), \]
and \eqref{eq:bound-on-T_v} follows.

From \eqref{eq:define-T_S} we obtain
\begin{eqnarray}
n & = & \sum_{v\in V(T_S)} |V(T_v)| \nonumber \\
   & \leq & \sum_{v\in C} |V(T_v)| +   \sum_{i=1}^k \sum_{w\in V(P_i)-C} |V(T_w)|.
        \label{eq:4}
\end{eqnarray}
By \eqref{eq:bound-on-T_v} we have $|V(T_v)|\leq r+1$ for all $v\in C$.
The vertices of $P_i$ are at distance $0,1,\ldots,\ell_i$ from $C$ in $T$, where
$\ell_i$ is the length of $P_i$. Hence, by \eqref{eq:bound-on-T_v} and $\ell_i\leq r$
we get
\begin{equation}
 \sum_{w\in V(P_i)-C} |V(T_w)|
       \leq \sum_{j=1}^{\ell_i} (r-j+1)
       \leq \sum_{j=1}^r (r-j+1)
       = \frac{1}{2}r(r+1).         \label{eq:bound2-on-T_v}
\end{equation}
In total we obtain
\begin{equation}
n \leq |C|(r+1) + \frac{k}{2}r(r+1).    \label{eq:total-bound-on-n}
\end{equation}
If $d$ is even, then $r=\frac{d}{2}$ and $C$ contains only one vertex. Hence
\[ n \leq (r+1) +  \frac{k}{2}r(r+1) = \frac{kr+2}{2}(r+1)  = \frac{(kd+4)(d+2)}{8}, \]
and the desired bound follows in this case. If $d$ is odd, then $r=\frac{d-1}{2}$
and $C$ contains exactly two vertices. Hence
\[ n \leq 2(r+1) + \frac{k}{2}r(r+1) = \frac{kr+4}{2}(r+1) = \frac{(kd-k+8)(d+1)}{8}, \]
and the desired bound follows also in this case.

Now assume that $T$ is a tree of diameter $d$ and metric dimension $k$ attaining
the bound. Then equality holds also in~\eqref{eq:total-bound-on-n}, in
\eqref{eq:4}--\eqref{eq:bound2-on-T_v}. So the paths $P_i$ share
no vertices other than central vertices, and each path has length $r$.
Moreover, equality in \eqref{eq:bound2-on-T_v} implies that for the vertices
$v$ of $P_i$, the trees $T_v$ have order $2,3,\ldots,r-1$, respectively, for each $i$.
Equality in \eqref{eq:total-bound-on-n} implies also that for each central vertex
$v$ the tree $T_v$ has $r+1$ vertices.
In total it follows that $T = HS_{d,k}$ if $d$ is even, and, if $d$ is odd, that $T = HS_{d,k,a}$ for
some $a\in \{0,1,\ldots,k\}$. It is easy to see that the trees $HS_{d,k,0}$ and
$HS_{d,k,k}$ have metric dimension $k+1$, so we conclude that $T=HS_{d,k,a}$
for some $a\in \{1,2,\ldots,k-1\}$.
\end{proof}

\subsection{Using tree decompositions}\label{sec:treedec}

We now generalize our result for trees to graphs with tree decompositions of given width and length. These results also generalize results of~\cite{part1} for interval graphs and permutation graphs (which have treelength at most~$1$ and~$2$, respectively~\cite{BFGR15}).

We first recall the definition of {\em tree decomposition} introduced by
Robertson and Seymour~\cite{seymour_minor_2}. We shall copy the
definition given by Dourisboure and Gavoille~\cite{DG05} which is
slightly lighter in terms of indices.

\begin{definition}[Tree decomposition \cite{DG05}]
  \label{def:tree_decomposition}
  Let $G$ be a graph. A {\em tree decomposition} of $G$ is a tree $T$
  whose vertices, called {\em bags}, are subsets of $V(G)$ such that
  the following properties are satisfied.

  \begin{enumerate}
    \item[(P1)] $\bigcup_{X \in V(T)} X = V(G)$,
    \item[(P2)] for every edge $e$ of $G$, there exists $X$ in $V(T)$
      such that both ends of $e$ are in $X$, and
    \item[(P3)] for $X,Y$ and $Z$ in $V(T)$, if $Y$ lies on the path in
      $T$ from $X$ to $Z$, then $X \cap Z \subseteq Y$.
  \end{enumerate}
\end{definition}

As mentioned in~\cite{DG05}, property (P3) of Definition~\ref{def:tree_decomposition} implies that, for any vertex $x$ in
$V(G)$, the set of bags containing $x$ induces a subtree of $T$. The classic {\em width} parameter of a tree decomposition is defined as $\max\{|X|-1 \,\colon X \in V(T)\}$. For any bag $X$ in $V(T)$, the
{\em diameter} of $X$ is the maximum distance $d_G(x,y)$ over every
pair of vertices $x$ and $y$ in $X$. (Note that here the distance is taken in $G$, and not in $G[X]$.) The {\em length} of a tree
decomposition is the largest diameter of a bag over every bag $X$ in $V(T)$~\cite{DG05}. The {\em treewidth} (respectively {\em treelength}) of a graph $G$ is the minimum width (resp. length) among all tree decompositions of $G$.

A tree decomposition is {\em reduced} if no bag is a subset of another
bag. One may easily check that any tree decomposition can be turned
into a reduced tree decomposition by removing the bags which are not
maximal with respect to inclusion and without altering the width and the length of the decomposition.

 A \emph{cutset} of a graph $G$ is a set of vertices in $G$
   whose removal increases the number of components. We shall prove the following theorem.

\begin{theorem}
Let $G$ be a graph of order $n$ and diameter $d$. Let $T$ be a reduced tree decomposition of $G$ of length $\ell$ and width $w$. If there is a resolving set of size $k$ in $G$, then
  \begin{equation*}
    n = \mathcal{O}(kd^2(2\ell+1)^{3w+1}).
%\leq (2k-1) (\ell+1)(2\ell+1)^{2w+1}(d+1).
  \end{equation*}
  \label{thm:main}
\end{theorem}
%\todo[inline]{Flo: on avait des arguments plus fins pour les chordaux/$2$-arbres, je sais pas si ça vaut le coup?}
%\todo[inline]{Aline:How interesting is this bound ? It would be nice to have examples of large chordal graphs with small resolving sets or to refine it...}

\begin{proof}
Let $G$ be a graph of diameter $d$ with a resolving set $S$ of size
$k$. Let $T$ be a tree decomposition of $G$ with length $\ell$ and
width $w$. The following claim is easily derived from the definition of a tree decomposition.

\begin{claim}
  \label{prop:cutset}
  Every bag $X$ which is not a leaf in $T$ is a cutset for $G$.
\end{claim}
%% \begin{proof}
%%   Let $X$ be a vertex of $T$ with degree greater than or equal to
%%   2. Then it is a cutting vertex in $T$. The graph $T - X$ is a forest
%%   with $t$ trees $T_1, T_2, \ldots, T_t$ where $t$ is greater than or
%%   equal to 2 (actually, $t$ is the degree of $X$ in $T$).
%%
%%   For every index $i$ between 1 and $t$, let $V_i$ be the union of the
%%   bags of $T_i$ minus $X$,
%%   \begin{equation*}
%%     V_i = \bigcup_{Y \in V(T_i)}Y \setminus X.
%%   \end{equation*}
%%   Since the tree decomposition is reduced, every $V_i$ is
%%   non-empty. Moreover, the subtree structure tells us that they are
%%   pairwise disjoint (if a vertex $x$ of $G$ is in two distinct such
%%   sets, it should also be in $X$ which is impossible).
%%
%%   We claim that $V_1,V_2,\ldots, V_t$ form a $t$-partition of $G -
%%   X$. For a contradiction, suppose there is an edge between two such
%%   sets. Without loss of generality, we may assume that there is $x_1$
%%   in $V_1$ and $x_2$ in $V_2$ such that $x_1$ and $x_2$ are adjacent
%%   in $G$. By property (P2) of Definition~\ref{def:tree_decomposition},
%%   there must be a bag $Y$ in $V(T)$ containing both $x_1$ and
%%   $x_2$. Since $x_1$ is in $V_1$, there is a bag $X_1$ in $T_1$
%%   containing $x_1$. Since the bags containing $x_1$ induce a subtree
%%   of $T$ and $X$ does not contain $x_1$, the bag $Y$ must be a vertex
%%   of $T_1$. Similarly, $Y$ must be a vertex of $T_2$ which is a
%%   contradiction. \end{proof}

For an easier reading of the following proofs, let us pick an
arbitrary root $X_r$ for $T$. For any bag $X$ in $V(T)$, we define the
subtree $T(X)$ as the subtree of $T$ induced by $X$ and all its descendants.

\begin{claim}
  \label{prop:cutsmall}
Let $X$ be a bag in $V(T)$ such that for every bag $Y$ in $T(X)$, the
set $Y \cap S$ is included in $X$. Let $A$ be the set defined as
  \begin{equation*}
    A = \bigcup_{Y \in V(T(X))} Y.
  \end{equation*}
  Then,
  \begin{equation*}
    \lvert A \rvert \leq (d+1)(2 \ell +1)^w.
  \end{equation*}
\end{claim}
\claimproof
For any vertex $x$ in $A$, every path from $x$ to an element of $S$
has to go through $X$ (this is implied by
Claim~\ref{prop:cutset}). Therefore, the distances from $x$  to the vertices of $S$ are completely determined by the distances
from $x$ to the vertices of $X$. Since $S$ is a resolving set, the
vertices in $A$ must all have a different distance vector to $X$. By
taking a specific vertex of $X$ as a pin point, the distance from
$x$ to this pin is at most $d$, and all other distances can only
differ from this distance by at most $\ell$. There are at most $w$
other vertices in $X$. Thus, the number of possible vectors is
smaller than or equal to
\begin{equation*}
(d+1)(2\ell+1)^w,
\end{equation*}
concluding the proof of Claim~\ref{prop:cutsmall}.~\smallqed

\medskip
For each vertex $v$, we call the bag in $T$ that contains $v$ and is at minimum distance from the root $X_r$ of $T$ the \emph{oldest bag in $T$ containing $v$}. We note that such a bag is uniquely defined because of the subtree structure and the properties of a tree decomposition $T$. For every vertex $s$ in the resolving set $S$, we denote the oldest bag in $T$ containing $s$ by $X_s$, and we call it the \emph{ancestor} of $s$ in $T$.

Let $T_S$ be the subtree of $T$ obtained by only considering the
ancestors of all $s$ in $S$ and the paths from them to the root
$X_r$. Any leaf of $T_S$ is the ancestor of some $s$ in $S$. As a
direct consequence, $T_S$ has at most $k$ leaves.  A \emph{thread}
   in a graph $G$ is a path all whose inner-vertices have degree~$2$
   in $G$.

\begin{claim}\label{prop:length}
Let $P$ be a thread of length $L$ in $T_S$. Let $X_0$ and $X_l$ be
the bags at both ends of $P$. Suppose that for every inner vertex $X$ of
$P$, the set $X \cap S$ is included in $X_0 \cup X_L$. Then,
\begin{equation*}
L \leq (\ell + 1)(2 \ell +1)^{2w+1}\left[d_G(X_0,X_l) + 1\right].
\end{equation*}
\end{claim}
\claimproof Let $\lambda$ be the distance in $G$ between the sets $X_0$ and
$X_L$. Let $x_0x_1\cdots x_{\lambda}$ be a shortest path in $G$
between $X_0$ and $X_L$ ($x_0$ is in $X_0$ and $x_{\lambda}$ is in
$X_L$). Note that every edge along the path $x_0x_1\cdots
x_{\lambda}$ must be in one of the bags along $P$.

If $x$ is in $X_i$ and $X_{i+t}$ for some $i,j$ along the thread, then it is in all
the bags in between $X_i$ and $X_{i+t}$ (by the connectivity condition). Notice that every path
between a vertex in $\bigcup_{z=i}^{i+t}X_z$ and a vertex in $S$ has
to go through $X_i$ or $X_{i+t}$. This means that all vertices in
$\bigcup_{z=i}^{i+t}X_z$ must have different distance vectors to
$X_i \cup X_{i+t}$.

These distances are bounded above by $2\ell$ since $x$ is in all the
bags along this thread. The distance to $x$ is at most $\ell$. There
are at most $2w$ vertices different from $x$ in $X_i \cup X_{i+t}$. We
may conclude that,
\begin{equation*}
  \left\lvert \bigcup_{z=i}^{i+t}X_z \right\rvert \leq (\ell + 1)(2\ell+1)^{2w}.
\end{equation*}

Since the tree decomposition is reduced, every bag $X_i$ must
contain a vertex which is not in any $X_j$ for $j$ between 0 and
$i-1$. We derive that the number of bags is smaller than the
number of vertices
\begin{equation}
    k+1 \leq (\ell + 1)(2\ell+1)^{2w}.
\label{eqn:bndk}
\end{equation}
  In other words, vertices cannot be in too many bags along the thread.

  Now, we shall prove that $L$ cannot be too big with respect to
  $\lambda$. For this, let us denote by $i_q$ the largest index of a
  bag containing $x_q$ for $q$ between 0 and $\lambda$,
  \begin{equation*}
    i_q = \max \{i \,\colon x_q \in X_i\}.
  \end{equation*}
  With the help of \eqref{eqn:bndk}, we may say that,
  \begin{equation*}
    i_0 \leq (\ell + 1)(2\ell+1)^{2w}.
  \end{equation*}
  Since $x_qx_{q+1}$ is an edge, vertex $x_{q+1}$ has to appear in a bag before index $i_q$. By using
  \eqref{eqn:bndk} successively, we obtain
  \begin{equation*}
    i_q \leq (q+1)(\ell + 1)(2\ell+1)^{2w}.
  \end{equation*}

  Substituting $q$ with $\lambda$ in the previous
  equation and noting that $i_{\lambda} = L$, we obtain that
   \begin{equation*}
    L \leq (\lambda+1)(\ell + 1)(2\ell+1)^{2w}.
   \end{equation*}
   Since $\lambda = \text{dist}_G(X_0,X_L)$, this concludes the
   proof of Claim~\ref{prop:length}.~\smallqed

\medskip

Let us now focus on $T_S$. Recall that its leaves are a subset of the
ancestors of vertices of $S$.  Let $A$ be the set of ancestors and $I$
be the set of inner vertices of degree at least $3$ in $T_S$ (note
that $I$ has cardinality at most $k-1$ since $T_S$ has at most $k$
leaves). We decompose $T_S$ into (not necesseraly disjoint) threads as
follows. From any vertex $X$ in $A \cup I$, consider the thread to the
closest vertex that is either in $I$ or in $A$ on the unique path from
$X$ to the root $X_r$. Each of these threads satisfies the conditions
of Claim~\ref{prop:length} and thus has size bounded above by
$(d+1)(\ell + 1)(2\ell+1)^{2w}$. Moreover we have at most $|A|+|I|$
such threads, and $|A|+|I|$ is at most $2k-1$. We can then conclude
that
\begin{equation*}
  \left\lvert V(T_S) \right\rvert \leq (2k-1)(d+1)(\ell + 1)(2\ell + 1)^{2w}.
\end{equation*}

For each bag $X$ in $T_S$, we may have removed from $T$ a part of the
subtree $T(X)$ verifying the hypothesis of Claim~\ref{prop:cutsmall}.
%Moreover, each remaining bag has size at most $w+1$.
In the end, the union of all the bags cannot be too large. We
then obtain the following upper bound on the order of $G$.
\begin{align*}
  \left\lvert V(G) \right\rvert & \leq  \left\lvert V(T_S) \right\rvert (d+1)(2 \ell +1)^w \\
  & \leq  (2k-1)(d+1)(\ell + 1)(2\ell + 1)^{2w}(d+1)(2 \ell +1)^w\\
  & = \mathcal{O}(kd^2(2\ell+1)^{3w+1}).
\end{align*}

This concludes the proof of Theorem~\ref{thm:main}.
\end{proof}

We immediately obtain some corollaries of Theorem~\ref{thm:main}. The first one is due to the fact that the treelength is trivially upper-bounded by the diameter; for graphs with constant treewidth, it implies the upper bound $n=\mathcal{O}(kd^{\mathcal{O}(1)})$.

\begin{corollary}\label{coro:bounded-treewidth}
Let $G$ be a graph of treewidth at most~$w$, diameter~$d$ and a resolving set of size~$k$. Then $$n=\mathcal{O}(k2^{3w}d^{3w+3}).$$ In particular, if $G$ is $K_4$-minor-free, then $$n=\mathcal{O}(kd^{9}).$$
\end{corollary}

We have another corollary for chordal graphs, based on the following observation and on the fact that chordal graphs have treelength~$1$~\cite{DG05}.

\begin{observation}
If $G$ is a chordal graph of treewidth~$w$ and with a resolving set of size~$k$, then $w\leq 3^k$.
\end{observation}
\begin{proof}
Let $v$ be a vertex of $G$, and let $x\in N[v]$. For any vertex $s$ in a resolving set of size~$k$ of $G$, there are at most three possible distance values for the distance $d(x,s)$, since any two vertices in $N[v]$ are at distance at most~$2$. Thus, there can be at most $3^k$ vertices in $N[v]$, which proves that $\Delta(G)\leq 3^k$. Now, a chordal graph of treewidth~$w$ must have a clique of size~$w+1$ (indeed, it is well-known that in any optimal treedecomposition of a chordal graph, each bag forms a clique), thus we have $w\leq \Delta(G)$.
\end{proof}

\begin{corollary}\label{coro:chordal}
If $G$ is a chordal graph with diameter~$d$, treewidth $w$ and a resolving set of size~$k$, then $$n=\mathcal{O}(kd^{2}3^{3w+1}).$$ Moreover, $$n=\mathcal{O}(d^22^{2^{\mathcal{O}(k)}}).$$
\end{corollary}

We do not know whether the bounds presented in this section are tight. We note that for interval graphs, which are chordal, it is known that a bound of the form $n=\mathcal{O}(dk^2)$ holds, and there are interval graphs for which $n=\Theta(dk^2)$~\cite{part1}. By Theorem~\ref{thm:trees}, there are trees that satisfy $n=\Theta(d^2k)$.

\section{Graphs of bounded distance-VC dimension}\label{sec:VC}

Let $\cH=(V,\mathcal E)$ be a hypergraph. A {\em test cover} of $\cH$ is a set of edges $\cC$ such that each vertex is covered by some edge of $\cC$ and for any pair $x$,$y$ of vertices there is an edge of $\cC$ containing exactly one vertex among $\{x,y\}$.
We denote by $\TC(\cH)$ the minimum size of a test cover of $\cH$. A hypergraph is \emph{twin-free} if for any two distinct vertices, there is at least one hyperedge containing exactly one of them. One can easily check that a hypergraph admits a test cover if and only if it is twin-free.
The \emph{projection} of $\cH$ on a set $X$ of vertices, is defined as %$\restriction{\cH}{X}
$\cH_{|X}:=\{e\cap X : e \in \cE\}$. A set of vertices $X$ is {\em shattered} in $\cH$ if $|\cH_{|X}|=2^{|X|}$.
% is a set of vertices $X$ such that for any subset $X'\subseteq X$, there is an edge $e$ of $\cH$ such that $e\cap X = X'$.
The maximum size of a shattered set in $\cH$ is the {\em VC dimension} of $\cH$, denoted by $\VC(\cH)$.

A {\em $2$-shattered set} in a hypergraph $\cH$ is a set $X$ such that for all $X'\subset X$ of size~$2$, there is a hyperedge $e$ such that $e\cap X=X'$. The {\em $2$-VC dimension} of $\cH$ is the maximum size of a $2$-shattered set in $\cH$. Clealry, the $2$-VC dimension of $\cH$ is at least as large as its VC dimension.

The {\em dual hypergraph} of a hypergraph $\cH$ is denoted $\cH^*$: it is the hypergrah whose vertices are the hyperedges of $\cH$, and vice-versa, and where the incidence relation is the same as in $\cH$. The {\em dual VC dimension} of $\cH$ is the VC dimension of the dual and is denoted by $\VC^*(\cH)$. We always have the following inequalities~\cite{A83}.
$$\log(\VC^*(\cH)) \leq \VC(\cH) \leq 2^{\VC^*(\cH)}.$$

The following standard lemma is crucial in the study of the VC dimension.

\begin{lemma}[Sauer-Shelah Lemma \cite{S72,Sh72}]\label{lem:Sauer}
If $\cH=(V,\cE)$  is a hypergraph and $X$, a subset of vertices, then $|\cH_{|X}|\leq |X|^{\VC(\cH)}+1$.
\end{lemma}

\subsection{A dichotomy theorem for test covers and VC dimension}

If $G$ is a graph, one can define the closed neighbourhood hypergraph ${\cH}_1(G)$ of $G$ that has vertex set $V(G)$ and edge set the set of closed neighbourhoods of vertices of $G$. An {\em identifying code} of $G$ is a test cover of $\cH_1(G)$, and the VC dimension of $G$ is often defined as the VC dimension of $\cH_1(G)$. A graph $G$ is \emph{twin-free} if $\cH_1(G)$ is twin-free. In \cite{BLLPT15}, the VC dimension and identifying codes are related by the following dichotomy result.

\begin{theorem}[\cite{BLLPT15}]\label{thm:dicho}
For every hereditary class of graphs $\cC$, either
\begin{enumerate}
\item for every $k\in \mathbb N$, there exists a graph $G_k\in \mathcal{C}$ with more than $2^k-1$ vertices and an identifying code of size $2k$, or
\item there exists $\varepsilon > 0$ such that no twin-free graph $G\in \cC$ with $n$ vertices has an identifying code of size smaller than $n^\varepsilon$.
\end{enumerate}
\end{theorem}

We show next that Theorem~\ref{thm:dicho} can be extended to test covers.

\begin{proposition}\label{prop:TCvsVC*}
If $\cH$ is a twin-free hypergraph, then $$|V|\leq (\TC(\cH))^{\VC^*(\cH)}+1.$$
\end{proposition}
\begin{proof}
Let $\cH^{*}$ be the dual hypergraph of $\cH$. Let $\cC$ be a test cover of $\cH$
 of size $\TC(\cH)$. We have $|\cH^{*}_{|\cC}|=|V|$ since otherwise two vertices
 of $V$ would belong to the same set of edges of $\cC$. Then by Lemma~\ref{lem:Sauer}, we have $|\cH^{*}_{|\cC}| \leq |\cC|^{\VC(\cH^{*})}=\TC(\cH)^{\VC^{*}(\cH)}+1$. Therefore, $|V|\leq \TC(\cH)^{\VC^{*}(\cH)}+1$.
\end{proof}

We can also prove the converse.

\begin{proposition}
Let $\cC$ be a class of hypergraphs that is stable by taking projections. If $\cC$ has unbounded dual VC dimension, then for any integer $k$, there exists a hypergraph $\cH$ in $\cC$ with $2^k-1$ vertices and a test cover of size $k$.
\end{proposition}
\begin{proof}
Notice first that for any $k$, $\cC$ contains a hypergraph with dual VC dimension exactly $k$. Indeed, assume  that $\cH$ is the hypergraph of $\cC$ with the smallest dual VC dimension $k'$ larger or equal to $k$. Then let $\mathcal{A}$ be a shattered set of hyperedges of size $k'$. Let $X$ be a set of vertices that shatters $\mathcal A$ (that is, for each subset of hyperedges of $\mathcal A$ there is a unique vertex of $X$ belonging to this subset). Remove one vertex $x$ of $X$ and let $\cH'=\cH_{|X-x}$. Then $\cH'$ belongs to $\mathcal{C}$ and $\VC^{*}(\cH')=k'-1$. Thus, by our assumptions, $k'=k$.

Now consider a hypergraph  of $\cC$ with dual  VC dimension $k$ and as before, let $\mathcal A$ be a shattered set of size $k$ and $X$ be a set of $2^k$ vertices such that there for each subset of hyperedges of $\mathcal A$ there is exactly one vertex of $X$ that is contained to exactly this subset of hyperedges. Let $x_0$ be the vertex of $X$ that is contained in no hyperedges and consider the hypergraph $\cH$ induced by $X\setminus \{x_0\}$. Notice first that $\cH$ belongs to $\cC$. By construction, $\cH$ has $2^{k}-1$ vertices and the set of hyperedges of $\mathcal A$ forms a test cover. Furthermore, any proper subset of hyperedges is not a test cover since the minimum size of a test cover among $2^k-1$ vertices is $k$.
\end{proof}

\subsection{Metric dimension, VC dimension and diameter}

In contrast to test covers, there is no direct relation between the VC dimension of $G$ and its metric dimension. Indeed, consider the family of line graphs. Any line graph has VC dimension at most~$4$.
Nevertheless, there is a line graph with more than $2^k$ vertices, diameter~$4$ and metric dimension at most~$k$. Indeed, consider the following graph. Take $k$ disjoint edges $\{e_1,...,e_k\}$ and $2^k-1$ disjoint edges $\{e'_I, I\subseteq \{1,...,k\}, I\neq \emptyset\}$ corresponding to the nonempty subsets of $\{e_1,...,e_k\}$.
For each edge $e'_I$, add $|I|$ edges between one endpoint of $e'_I$ (always the same one) and all the endpoints of $e_i$ for $i\in I$ (again, choose always the same endpoint for $e_i$).
Let $G$ be the line graph of this graph. The graph $G$ has $k+2^k-1+\sum_{i=1}^{k}i{k \choose i}$ vertices and diameter $4$. Moreover, the set $S$ of vertices corresponding to the edges $\{e_1,...,e_k\}$ forms a resolving set. Indeed, a vertex corresponding to an edge $e'_I$ has distance~$2$ to $e_i$ if $i\in I$ and $4$ otherwise. A vertex corresponding to an edge between $e'_I$ and $e_i$ (with $i\in I$) has distance~$1$ to $e_i$, 2 to $e_j$ when $j\in I$ and $4$ otherwise. Therefore all the edges have unique distance vector to $S$.

However, there is such a relation when we consider the {\em distance-VC dimension}, introduced by Bousquet and Thomass\'e~\cite{BT15}. The {\em distance hypergraph} of $G$ is the hypergraph $\cH(G)$ with vertex set $V$ and for all $\ell$, all the balls of radius $\ell$. The {\em distance-VC dimension} of $G$, $\dVC(G)$, is the VC dimension of $\cH(G)$. The {\em dual distance-VC dimension} of $G$, denoted $\dVC^*(G)$, is the VC dimension of $\cH(G)^*$. Similarly, the \emph{(dual) $2$-distance VC dimension} of $G$ is the $2$-VC dimension of $\cH(G)$ ($\cH(G)^*$, respectively).

We first give a relation between test covers in $\cH(G)$ and the metric dimension of $G$.

\begin{proposition}\label{prop:MDvsTC}
If $G$ is a graph of diameter $d$ and metric dimension $k$, then we have the following.
$$\frac{\TC(\cH(G))-1}{d} \leq k \leq \TC(\cH(G)).$$
\end{proposition}

% [check the -1 : the balls of radius $d$ are not really interesting but we need to cover all the vertices...]

\begin{proof}
Let $T$ be a test  cover of $\cH(G)$. Then the set of centers of the balls corresponding to the  hyperedges of $T$ form a resolving set.
Indeed, let $x,y\in V$ and assume without loss of generality that there exists
$B\in T$ such that $x\in B$ and $y\notin B$. Let $v$ be the center of $B$ and let $r$ be its radius. Then $v$ resolves $\{x,y\}$ since $d(v,x) \leq r< d(v,y)$. This shows that $k \leq \TC(\cH(G))$.

Now let  $R$ be a resolving set and let $T$ be the set of balls centered in
vertices of $R$ for all radius from $0$ to $d-1$ plus any ball with radius $d$ . Then $T$ is a test cover of $\cH(G))$.
Indeed let $x,y\in V$ and let $z\in R$ such that $d(z,x)\neq d(z,y)$. Assume
without loss of generality that $d(z,x)<d(z,y)$. Then $d(z,x)< d$ and the ball centered in $z$
with radius $d(z,x)$ distinguishes $x$ and $y$. Thus, since any vertex is covered by the ball of radius $d$, the test cover $T$ has size $d|R| +1$.
%One can form a test cover by taking all the balls centered in some vertex of the resolving set.
\end{proof}

We deduce the following.

\begin{proposition}\label{md-vcdim} If $G$ is a graph of order $n$ with diameter $d$ and a resolving set of size $k$, then
$$n\leq (dk+1)^{\dVC^*(G)}+1.$$
\end{proposition}
\begin{proof}
By Proposition~\ref{prop:TCvsVC*}, $n\leq (TC(\cH(G)))^{\dVC^*(G)}+1$. Then, by Proposition~\ref{prop:MDvsTC}, we have 
$\TC(\cH(G))\leq kd+1$. 
\end{proof}

% \todo[inline]{Flo: add here the corollaries for graphs of bounded rankwidth and with no $K_m$-minor (using the logarithmic relation between distance-VC dim of the hypergraph and its dual). The one for $K_m$-minor free ones will be improved in the next sub-section.}
% \todo[inline]{Flo: do we need the following proposition? I think not.}

 Proposition~\ref{md-vcdim} is useful when one can bound the dual distance-VC dimension of a graph. The next proposition gives a relation between $\dVC$ and $\dVC^*$.
%  When dealing with distance-VC dimension, we have other relations with the dual version than the logarithmic one:

\begin{proposition}\label{prop:md-vcdim-2} If $G$ is a graph of diameter $d$, then
$$\frac{1}{\log \dVC(G)}(\dVC(G)-\log d) \leq \dVC^*(G) \leq d \cdot \dVC(G).$$
\end{proposition}
%\todo[inline]{It should be d+1}
\begin{proof}
For the first inequality, let $k$ denote $\dVC(G)$. Let $S$ be a
shattered set of $\cH(G)$ of size $k$. For each subset $X$ of $S$,
there exists a ball $B$ such that $B\cap S=X$. Let $\mathcal B$ be the
set of those balls.  Among all the radii used in $\mathcal B$, let us
consider the most used $\ell$ and let $\mathcal{B}_{\ell}$ be the set
of balls of $\mathcal{B}$ of radius $\ell$. We have $|\mathcal
B_{\ell}| \geq \frac{|\mathcal B|}{d}\geq
\frac{2^{k}}{d}$. Considering the hypergraph $\cH_{\ell}(G)$ formed by
all balls of $G$ of radius $\ell$, we have $|\cH_{\ell}(G)_{|S}|\geq
|\mathcal{B}_{\ell}|\geq \frac{2^{k}}{d}$ and then, by
Lemma~\ref{lem:Sauer}, $\frac{2^{k}}{d}\leq k^{\VC(\cH_{\ell}(G))}$
which implies that $\frac{k-log(d)}{log(k)}\leq
\VC(\cH_{\ell}(G))$. Now since $\cH_{\ell}$ is isomorphic to its dual
we have $\frac{k-log(d)}{log(k)}\leq \VC^{*}(\cH_{\ell})\leq
\VC^{*}(\cH(G))=\dVC^{*}(G)$.

For the second inequality, let $S$ be a shattered set of $\cH(G)^{*}$ of size $\dVC^*(G)$. Let $\ell$ be the most used radius in $S$ and let $S_{\ell}$ be the set of balls of $S$ of radius $\ell$. Let $\cH_{\ell}(G)$ be the the hypergraph formed by all balls of $G$ of radius $\ell$. Notice that $\cH_{\ell}(G)$ is isomorphic to its dual $\cH_{\ell}(G)^{*}$ and then $\VC(H_{\ell}(G))=\VC(H_{\ell}(G)^{*})$. Observe now that $S_{\ell}$ is a shattered set of $\cH_{\ell}(G)^{*}$ and since $|S_{\ell}|\geq \frac{\dVC^*(G)}{d}$ we have $\dVC(G)\geq \VC(H_{\ell}(G))\geq \frac{\dVC^*(G)}{d}$.
% Then  consider the radius $\ell$ that is the most used, then it gives a shattered set in the dual of $\cH_\ell(G)$ which is autodual and that is a shattered set of $\cH(G)$.
\end{proof}

Bousquet and Thomassé proved that graphs of bounded rankwidth\footnote{We do not define this concept here, since we barely use it, and refer the reader to~\cite{BT15} instead. Note that any graph of bounded treewidth or cliquewidth also has bounded rankwidth.} and $K_t$-minor free graphs have bounded distance $2$-VC dimension (and thus, bounded distance-VC dimension).

\begin{theorem}[\cite{BT15}]\label{thm:2vc_bounded}
A $K_t$-minor-free graph has distance $2$-VC dimension at most $t-1$. The distance $2$-VC dimension of a graph with rankwidth~$r$ is at most $3\cdot 2^{r}+2$.
\end{theorem}

Since the distance $2$-VC dimension is always larger than the distance-VC dimension and using Proposition~\ref{prop:md-vcdim-2}, we have the following corollaries of Proposition~\ref{md-vcdim}.

\begin{corollary}\label{cor:minor-rankwidth-bounds}
Let $G$ be a graph of order $n$, diameter $d$ and with a resolving set of size $k$. If $K_t$ is not a minor of $G$, then
$$n\leq (dk+1)^{d(t-1)}+1.$$

If $G$ has rankwidth at most $r$, then
$$n\leq (dk+1)^{d(3\cdot 2^{r}+2)}+1.$$
\end{corollary}

\subsection{The dual $2$-distance VC dimension and $K_t$-minor free graphs}

In this section, we improve the bound of Corollary~\ref{cor:minor-rankwidth-bounds} for $K_t$-minor-free graphs.

\begin{theorem}\label{thm:minorfree}
If $G$ is a $K_t$-minor-free graph of diameter $d$ and order $n$, with a resolving set of size $k$, then $n\leq (dk+1)^{t-1}+1$.
\end{theorem}

To prove Theorem~\ref{thm:minorfree}, we combine Proposition~\ref{md-vcdim} with the following theorem, which is a ``dual'' version of Theorem~\ref{thm:2vc_bounded}. We denote the length of a path $P$ by $\ell(P)$.

\begin{theorem}\label{thm:2-VC-minorfree}
If the dual distance $2$-VC dimension of a graph $G$ is at least $t$, then $K_t$ is a minor of $G$.
\end{theorem}
%\begin{corollary}
%If $G$ is $K_m$-minor free, the $2$-VC dimension of the dual of $\cH(G)$ is at most $d-1$.
%\end{corollary}
\begin{proof}
To prove Theorem~\ref{thm:2-VC-minorfree}, we adapt the proof of~\cite{BT15} for distance $2$-VC dimension to the dual distance $2$-VC dimension, and prove the following.

Let $\{(v_1,r_1),\ldots,(v_t,r_t)\}$ be a $2$-shattered set in the dual of $\cH(G)$. Then, for all $i,j$, there exists $x_{ij}$ such that:
\begin{itemize}
\item $d(x_{ij},v_i)\leq r_i$
\item $d(x_{ij},v_j)\leq r_j$
\item $d(x_{ij},v_k)>r_k$ if $k\notin\{i,j\}$
\end{itemize}

For any such $x_{ij}$, a path formed by a path $P$ between $v_i$ and $x_{ij}$ and a path $P'$ between $x_{ij}$ and $v_j$ such that $\ell(P)\leq r_i$ and $\ell(P')\leq r_j$ is called a {\em good $ij$-path}. For a path $P$ and two vertices $x,y$ in $P$, we denote by $P[x,y]$ the subpath of $P$ between $x$ and $y$.

\begin{claim}\label{lem:four_distincts}
If $i$, $j$, $k$, $l$ are distinct and $P_{ij}$, $P_{kl}$ are two good paths, then $P_{ij}\cap P_{kl}=\emptyset$.
\end{claim}
\claimproof
Let $P_i:=P_{ij}[v_i,x_{ij}]$ be the path from $v_i$ to $x_{ij}$ and let $P_k:=P_{kl}[v_k,x_{kl}]$ be the path from $v_k$ to $x_{kl}$. Suppose for contradiction that there exists $u\in P_{ij}\cap P_{kl}$. Assume  without loss of generality that  $u\in P_{i}\cap P_{k}$ and that $\ell(P_{k}[u,x_{kl}])\leq \ell(P_{i}[u,x_{ij}])$. Then,  since $\ell(P_{i}) = \ell(P_{i}[v_i,u])+\ell(P_{i}[u,x_{ij}]) \leq r_i$, we have $d(v_i,x_{kl})\leq \ell(P_{i}[v_i,u])+\ell(P_{i}[u,x_{kl}])\leq r_i$ which is a contradiction.~\smallqed

\medskip

\begin{claim}\label{lem:part}
If $i$, $j$, $k$ are distinct and $P_{ij}$, $P_{ik}$ are two good paths that intersect in $z$, then $x_{ij}$ and $x_{ik}$ cannot both be in the part of $P_{ij}$ (resp. $P_{ik}$) that is between $v_i$ and $z$. Hence, at least one of $x_{ij}\in P_{ij}[z,v_j]$ or $x_{ik}\in P_{ik}[z,v_k]$ is true.
\end{claim}
\claimproof
Suppose, to the contrary, that both $x_{ij}$ and $x_{ik}$ are between $z$ and $v_i$ and assume without loss of generality that $\ell(P_{ij}[z,x_{ij}])\leq \ell(P_{ik}[z,x_{ik}])$. Then we have $$d(v_k,x_{ij})\leq \ell(P_{ik}[v_k,z])+\ell(P_{ij}[z,x_{ij}])\leq \ell(P_{ik}[v_k,z])+ \ell(P_{ik}[z,x_{ik}])= \ell(P_{ik}[v_k,x_{ik}])\leq r_k$$ contradicting the fact that $d(v_k,x_{ij})>r_k$.~\smallqed

\medskip

\begin{claim}\label{lem:three_distincts}
If $i$, $j$, $k$ are distinct and $P_{ij}$, $P_{ik}$ and $P_{jk}$ are three good paths, then $P_{ij}\cap P_{ik}\cap P_{jk}=\emptyset$.
\end{claim}
\claimproof
Let $z\in P_{ij}\cap P_{ik}\cap P_{jk}$. Assume without loss of generality
that $$\ell(P_{ij}[z,x_{ij}])=min(\ell(P_{ij}[z,x_{ij}]),\ell(P_{ik}[z,x_{ik}]),\ell(P_{jk}[z,x_{jk}]))$$ Assume furthermore that  $x_{ij}\in P_{ij}[v_i,z]$. By Claim~\ref{lem:part}, $x_{ik}\in P_{ik}[z,v_k]$ and $x_{jk}\in  P_{jk}[z,v_j]$. Now we have $$ d(v_k,x_{ij})\leq
\ell(P_{jk}[v_k,z])+\ell(P_{ij}[z,x_{ij}]) \leq
\ell(P_{jk}[v_k,z])+\ell(P_{jk}[z,x_{jk}])= \ell(P_{jk}[v_k,x_{jk}])\leq
r_k $$ contradicting the fact that $d(v_k,x_{ij})>r_k$.~\smallqed

\medskip

For all $x\in V$, we give label $i$ to $x$ if there exists two good paths $P_{ij}$ and $P_{ik}$ that intersects in $x$. Note that $v_i$ has label $i$.

\begin{claim}
For all $x\in V$, $x$ has at most one label.
\end{claim}
\claimproof
Let $P_{ij}$ and $P_{kl}$ be two good paths containing $x$. By
Claim~\ref{lem:four_distincts} we have $\{i,j\}\cap \{l,k\}\neq \emptyset$.
Assume without loss of generality that $i=k$. Assume now that there exists a
third good path $P_{mn}$ containing $x$. We show that  $i\in \{m,n\}$.
Suppose, to the contrary, that $i\neq m$ and $i\neq n$. Since $x\in P_{ij} \cap P_{mn}$, by Claim
\ref{lem:four_distincts}, either $m=j$ or $n=j$ (say $m=j$). Now since  $x\in
P_{il} \cap P_{mn}$, we have that $n=l$. But then $x\in P_{ij}\cap P_{il} \cap
P_{jl}$ which is in contradiction with Claim~\ref{lem:three_distincts}. So every
good path containing $x$ is a good path from  $v_i$ and then $x$ has only
label $i$.~\smallqed

\medskip

Let $C_i$ be the set of vertices that are labeled $i$. Since $v_i$ has label $i$, $C_i$ is non-empty.

\begin{claim}
For all $i\leq d$, $C_i$ induces a connected subgraph.
\end{claim}
\claimproof
We will prove that for each vertex $u\in C_i$ , there exists a path in
$C_i$ from $u$ to $v_i$. Assume that $u \in P_{ij}\cap P_{il}$. By Claim~\ref{lem:part} either $x_{ij} \in P_{ij}[v_i,u]$ or $x_{il} \in
P_{il}[v_i,u]$. Assume without loss of generality that $x_{ij} \in
P_{ij}[v_i,u]$. By definition of $x_{il}$, $r_j< d(v_j,x_{il})\leq
\ell(P_{ij}[u,v_j])+\ell(P_{il}[u,x_{il}])$ and since
$\ell(P_{ij}[u,v_j])+\ell(P_{ij}[x_{ij},u]) \leq r_j$, we have
$\ell(P_{il}[u,x_{il}])> \ell(P_{ij}[x_{ij},u])$.  Then since
$\ell(P_{il}[v_i,u])+\ell(P_{il}[u,x_{il}])=\ell(P_{il}[v_i,x_{il}])\leq
r_i$ we have $\ell(P_{il}[v_i,u])+\ell(P_{ij}[u,x_{ij}])\leq r_i$. Thus
$P_{il}[v_i,u]\cup P_{ij}[u,x_{ij}]\cup P_{ij}[x_{ij},v_j]$ is a good
$ij$-path. We conclude that all vertices of $P_{il}[v_i,u]$ have label $i$
(that is, $P_{il}[v_i,u]\subseteq C_i$).~\smallqed

\medskip

Assume now that $\{(v_1,r_1),\ldots,(v_t,r_t)\}$ is a $2$-shattered set in the dual of $\cH(G)$. Then the sets $C_i$ form  non-empty connected disjoint sets of vertices and there are disjoint paths between any pair of such sets. Thus there is a minor $K_t$, completing the proof of Theorem~\ref{thm:2-VC-minorfree}.
\end{proof}

\subsection{Outerplanar graphs}\label{sec:outerplanar}

Outerplanar graphs are $K_4$-minor-free and have treewidth at most~$2$. Hence, by Theorem~\ref{thm:minorfree}, $n=\mathcal{O}(d^3k^3)$ and by Corollary~\ref{coro:bounded-treewidth}, $n=\mathcal{O}(d^9k)$ . We will improve these bounds using a different method.

\begin{theorem}\label{prop:boundouterplanar}
If $G$ is an outerplanar graph with diameter $d$ and a resolving set of size~$k$, then $G$ has order at most $2kd^2-2d^2+d+1=\mathcal{O}(kd^2)$.
\end{theorem}
\begin{proof}
Let $S$ be a resolving set of $G$ of size $k$ and let $s_1\in S$. We consider a circular layout of $G$, that is, a planar representation of $G$ with all the vertices lying on the boundary of a circle $\cC$ (it is not difficult to see that such a layout exists, see~\cite{ST99}). The vertices of $G$ can be naturally ordered following $\cC$ and starting by $s_1$. We denote this order by $<$.

\begin{claim}\label{clm:intersect}
Let $x<y<z<t$ be four vertices of $G$. Let $P_1$ be a path from $y$ to $t$ and $P_2$ be a path from $x$ to $z$. Then $P_1$ and $P_2$ must intersect.
\end{claim}
\claimproof
Indeed, the drawing of the path $P_2$ cuts the disk formed by $\cC$ into two disjoint components and the vertices $y$ and $t$ are not in the same component. Therefore, the drawing of the path $P_1$ must intersect $P_2$, and since the representation is planar, it must be on a vertex.~\smallqed

\medskip

For each $1\leq i \leq d$, we define $L_i$ to be the set of vertices at distance exactly $i$ of $s_1$. The following claim is key to our proof.

\begin{claim}\label{lem:orderrespect}
Let $i\in\{1,...,d\}$.
Let $s\in S$ and $y$ a vertex of $L_i$ that minimizes the distance between $s$ and vertices of $L_i$.
Let $u$ and $v$ be two vertices of $L_i$. If $y<u<v$ or $v<u<y$, then, $d(s,u)\leq d(s,v)$.
\end{claim}
\claimproof
We assume that $s_1<y<u<v$ (the other case is symmetric).

We first prove that $d(y,u)\leq d(y,v)$. Let $P_1$ be a shortest path from $y$ to $v$ and $P_2$ be a shortest path from $u$ to $s_1$. By Claim~\ref{clm:intersect}, $P_1$ and $P_2$ must intersect in some vertex $z$. Since $P_2$ is a shortest path from $u$ to $s_1$, it has length $i$ and $d(z,s_1) \leq i$, and so $z\in L_j$ with $j\leq i$. Furthermore, we have $d(z,v)\geq i-j=d(z,u)$. Let $P$ be the path from $y$ to $u$ that consists of the subpath of $P_1$ from $y$ to $z$, followed by the subpath of $P_2$ from $z$ to $u$. Since $d(z,u)\leq d(z,v)$, the path $P$ is not longer than $P_1$ and thus $d(y,u)\leq d(y,v)$.

This proves the claim when $s\in L_i$. Assume now that $s\in L_j$ and that $j<i$. Let $P_1$ be a path formed by the union of a shortest path $P_{1,1}$ from $y$ to $s$ and a shortest path $P_{1,2}$ from $s$ to $v$. Let $P_2$ be a shortest path from $u$ to $s_1$. By Claim~\ref{clm:intersect}, $P_1$ and $P_2$ must intersect in $z$ and as before, $z\in L_k$ with $k\leq i$. If $z$ belongs to $P_{1,1}$, that is, to a shortest path between $y$ and $s$, then the path from $s$ to $u$ following $P_{1,1}$ until $z$ and then $P_2$ until $u$ is shorter than $P_{1,1}$ (since $d(z,u) \leq d(z,y)$). Hence $d(s,u)\leq d(s,y)$ but $y$ is minimizing the distance between $s$ and a vertex of $L_i$. Therefore, $d(s,y)=d(s,u)\leq d(s,v)$. Otherwise, $z$ must belong to $P_{1,2}$, a shortest path between $s$ and $v$. Then the path $P$ from $s$ to $u$ that follow $P_{1,2}$ until $z$ and then $P_2$ until $u$ is shorter than $P_{1,2}$. Indeed, since $z\in L_k$, we have $d(z,u)=i-k$ and  $d(z,v)\geq i-k$. Thus $d(s,u)\leq d(s,v)$.

Assume finally that $s\in L_j$ with $j>i$. We have $d(s,y)=j-i$ (indeed, a shortest path from $s$ to $s_1$ must pass by a vertex $y'\in L_i$ and then $d(s,y')=j-i$). Let $P_1$ be a path formed by the union between a shortest path $P_{1,1}$ from $y$ to $s$ and a shortest path $P_{1,2}$ from $s$ to $v$. Let $P_2$ a shortest path from $u$ to $s_1$. Again, $P_1$ and $P_2$ must intersect in $z\in L_k$, with $k\leq i$. Since $P_{1,1}$ is a path of length $j-i$ between $L_i$ and $L_j$, all the vertices of $P_{1,1}$ are in a layer $L_{j'}$ with $i\leq j'\leq j$. It is not possible to have $z=y$ since in $P_2$ there is exactly one vertex by $L_{j'}$ for $j'\leq i$, and $u\neq y$ is this vertex for $j'=i$. Hence $z$ is in $P_{2,2}$. It means that there is a vertex $z'\in L_i$ on the path from $s$ to $z$: indeed when going from $L_j$ to $L_k$ a path must intersect all the layers between $L_k$ and $L_j$. We choose for $z'$ the first vertex of $L_i$ we meet on $P_{1,2}$ going from $s$ to $v$. If $z'<u<v$ then as in the first case of the proof, $d(z',u)\leq d(z',v)$ and so $d(s,u)\leq d(s,v)$ and we are done. Otherwise, we have $s_1 <y<u<z'$ and there is a path from $y$ to $z'$ (the path $P_1$ stopped in $z'$) that is not intersecting a path from $s_1$ to $u$, which contradicts Claim~\ref{clm:intersect}. This completes the proof of Claim~\ref{lem:orderrespect}.~\smallqed

\medskip

We can now finish the proof of Theorem~\ref{prop:boundouterplanar}. By Claim~\ref{lem:orderrespect}, each vertex $s\neq s_1$ of $S$ partitions the vertices of $L_i$ with respect to the order $<$ into at most $2d+1$ parts such that two vertices belonging to the same part have the same distance to $s$. Hence, together, the vertices of $S\setminus \{s_1\}$ partition $L_i$ into at most $2d(k-1)+1$ parts and the distance to $S$ of each vertex of $L_i$ is determined by its position in the partition. Hence, there is at most one vertex in each part, and thus $|L_i|\leq 2d(k-1)+1$. Finally, the total number of vertices of $G$ is at most $1+\sum_{i=1}^d|L_i|\leq 2d^2(k-1)+d+1$. This completes the proof of Theorem~\ref{prop:boundouterplanar}.
\end{proof}

%\todo[inline]{Aline: Note that we can slightly improve the constant $2$ to $7/4$. Indeed, a vertex $s$ in $L_j$ will actually divide $L_i$ into at most $2(d-|i-j|)+1$ parts and thus make in total at most $\frac{7}{4}d^2$ parts (the worse case being $s\in L_{i/2}$).}

We now show that Theorem~\ref{prop:boundouterplanar} is tight, up to a constant factor. For two integers~$d,k\geq 2$, let $O_{d,k}$ be the outerplanar graph constructed as follows. First, for some integer $i$, we define a graph $H_i$ as follows. Consider a cycle $C$ of length~$2i+1$, where $x$ is a distinguished vertex of $C$. To any vertex $v$ of $C$ at distance $j\geq 1$ of $x$ in $C$, we attach a path of length $i-j+1$ to $v$, and to one of the two vertices at distance~$i$ of $x$ in $C$, we attach a second leaf. Now, $O_{d,k}$ is built from $k-1$ copies of $H_{\lfloor d/2\rfloor-1}$ and one copy of $H_{\lceil d/2\rceil-1}$ identified at $x$, with an additional path of length $\lfloor d/2\rfloor$ attached to $x$. (Note that we may optionally add chords to the cycles in $O_{d,k}$, as long as the outerplanarity and the distances from each vertex having two leaves attached are preserved.) The order of $O_{d,k}$ is $\frac{d+2}{2}+k\left(2\sum_{i=1}^{d/2}i-1\right)$ when $d$ is even, and $\frac{3d+3}{2}+k\left(2\sum_{i=1}^{\lfloor d/2\rfloor}i-1\right)$ when $d$ is odd; this is $(\frac{1}{4}+o(1))kd^2$. See the graph of Figure~\ref{fig:extr-outerplanar} for an illustration of $O_{7,3}$ and $O_{8,3}$.

\begin{proposition}
Let $d,k\geq 2$ be two integers. The outerplanar graph $O_{d,k}$ has diameter $d$, metric dimension $k$ and order $(\frac{1}{4}+o(1))kd^2$.
\end{proposition}
\begin{proof}
The values of the diameter and the order follow from the definition. To see that the metric dimension is~$k$, consider the $k$ vertices that have two neighbours of degree~$1$. In order for these two neighbours to be distinguished, one of them needs to be in any resolving set. Now, we pick exactly one of them and repeat this for every such pair; we obtain a set $S$ of $k$ vertices. We claim that $S$ is a resolving set. Let $S=\{s_1,\ldots,s_k\}$ and let us call $C_i$ the component of $O_{d,k}-x$ containing $s_i$. (There are $k+1$ components in $O_{d,k}-x$, with one of them isomorphic to a path and not containing any vertex of $S$.) Any two vertices $u$ and $v$ from different components of $O_{d,k}-x$ are distinguished, since at least one of them (say $u$) has a vertex $s_i$ of $S$ in its component $C_i$, and $d(u,s_i)<d(v,s_i)$. Within a component $C$ of $O_{d,k}-x$, vertices with distinct distances to $x$ are distinguished by the vertices of $S$ not in $C$. Finally, vertices at the same distance of $x$ in a component $C_i$ are distinguished by $s_i$.
\end{proof}

\begin{figure}[ht!]
  \centering
\scalebox{0.8}{
  \subfigure[The graph $O_{7,3}$.]{
\begin{tikzpicture}[join=bevel,inner sep=0.6mm,line width=0.8pt, scale=0.3]
    \begin{scope}[rotate=90]

        \foreach \i in {1,...,4}{

          \ifnum \i<4
          \path (5+3*\i,0) node[draw, shape=circle] (a-\i-0) {};

          \pgfmathtruncatemacro{\k}{4-\i}
          \foreach \j in {1,...,\k}{
            \path (5+3*\i,-2*\j) node[draw, shape=circle] (a-p\i-\j) {};

            \ifnum \j=1
            \draw[line width=0.5pt] (a-p\i-\j)--(a-\i-0);
            \fi

            \ifnum \j>1
            \pgfmathtruncatemacro{\jj}{\j-1}
            \draw[line width=0.5pt] (a-p\i-\j)--(a-p\i-\jj);
            \fi
          }
          \fi

          \ifnum \i=4
          \path (5+3*\i,0) node[draw, shape=circle,fill] (a-\i-0) {};
          \fi

          \ifnum \i>1
          \pgfmathtruncatemacro{\ii}{\i-1}
          \draw[line width=0.5pt] (a-\i-0)--(a-\ii-0);
          \fi

        }

        \foreach \i in {1,2,3}{

          \path (5+3*\i-1,3) node[draw, shape=circle] (a-\i-1) {};

          \ifnum \i<3
          \draw[line width=0.5pt,dashed] (a-\i-1)--(a-\i-0);
          \fi
          \ifnum \i=3
          \draw[line width=0.5pt] (a-\i-1)--(a-\i-0);
          \fi

          \ifnum \i>1
          \pgfmathtruncatemacro{\ii}{\i-1}
          \draw[line width=0.5pt] (a-\i-1)--(a-\ii-1);
          \draw[line width=0.5pt,dashed] (a-\i-1)--(a-\ii-0);
          \fi

          \pgfmathtruncatemacro{\k}{4-\i}
          \foreach \j in {1,...,\k}{
            \path (5+3*\i-1,3+2*\j) node[draw, shape=circle] (a-pt\i-\j) {};

            \ifnum \j=1
            \draw[line width=0.5pt] (a-pt\i-\j)--(a-\i-1);
            \fi

            \ifnum \j>1
            \pgfmathtruncatemacro{\jj}{\j-1}
            \draw[line width=0.5pt] (a-pt\i-\j)--(a-pt\i-\jj);
            \fi
          }

        }

    \end{scope}

    \begin{scope}[rotate=210]

        \foreach \i in {1,...,3}{

          \ifnum \i<3
          \path (5+3*\i,0) node[draw, shape=circle] (b-\i-0) {};

          \pgfmathtruncatemacro{\k}{3-\i}
          \foreach \j in {1,...,\k}{
            \path (5+3*\i,-2*\j) node[draw, shape=circle] (b-p\i-\j) {};

            \ifnum \j=1
            \draw[line width=0.5pt] (b-p\i-\j)--(b-\i-0);
            \fi

            \ifnum \j>1
            \pgfmathtruncatemacro{\jj}{\j-1}
            \draw[line width=0.5pt] (b-p\i-\j)--(b-p\i-\jj);
            \fi
          }
          \fi

          \ifnum \i=3
          \path (5+3*\i,0) node[draw, shape=circle,fill] (b-\i-0) {};
          \fi

          \ifnum \i>1
          \pgfmathtruncatemacro{\ii}{\i-1}
          \draw[line width=0.5pt] (b-\i-0)--(b-\ii-0);
          \fi

        }

        \foreach \i in {1,2}{

          \path (5+3*\i-1,3) node[draw, shape=circle] (b-\i-1) {};

          \ifnum \i<2
          \draw[line width=0.5pt,dashed] (b-\i-1)--(b-\i-0);
          \fi
          \ifnum \i=2
          \draw[line width=0.5pt] (b-\i-1)--(b-\i-0);
          \fi

          \ifnum \i>1
          \pgfmathtruncatemacro{\ii}{\i-1}
          \draw[line width=0.5pt] (b-\i-1)--(b-\ii-1);
          \draw[line width=0.5pt,dashed] (b-\i-1)--(b-\ii-0);
          \fi

          \pgfmathtruncatemacro{\k}{3-\i}
          \foreach \j in {1,...,\k}{
            \path (5+3*\i-1,3+2*\j) node[draw, shape=circle] (b-pt\i-\j) {};

            \ifnum \j=1
            \draw[line width=0.5pt] (b-pt\i-\j)--(b-\i-1);
            \fi

            \ifnum \j>1
            \pgfmathtruncatemacro{\jj}{\j-1}
            \draw[line width=0.5pt] (b-pt\i-\j)--(b-pt\i-\jj);
            \fi
          }

        }

    \end{scope}

    \begin{scope}[rotate=330]

        \foreach \i in {1,...,3}{

          \ifnum \i<3
          \path (5+3*\i,0) node[draw, shape=circle] (c-\i-0) {};

          \pgfmathtruncatemacro{\k}{3-\i}
          \foreach \j in {1,...,\k}{
            \path (5+3*\i,-2*\j) node[draw, shape=circle] (c-p\i-\j) {};

            \ifnum \j=1
            \draw[line width=0.5pt] (c-p\i-\j)--(c-\i-0);
            \fi

            \ifnum \j>1
            \pgfmathtruncatemacro{\jj}{\j-1}
            \draw[line width=0.5pt] (c-p\i-\j)--(c-p\i-\jj);
            \fi
          }
          \fi

          \ifnum \i=3
          \path (5+3*\i,0) node[draw, shape=circle,fill] (c-\i-0) {};
          \fi

          \ifnum \i>1
          \pgfmathtruncatemacro{\ii}{\i-1}
          \draw[line width=0.5pt] (c-\i-0)--(c-\ii-0);
          \fi

        }

        \foreach \i in {1,2}{

          \path (5+3*\i-1,3) node[draw, shape=circle] (c-\i-1) {};

          \ifnum \i<2
          \draw[line width=0.5pt,dashed] (c-\i-1)--(c-\i-0);
          \fi
          \ifnum \i=2
          \draw[line width=0.5pt] (c-\i-1)--(c-\i-0);
          \fi

          \ifnum \i>1
          \pgfmathtruncatemacro{\ii}{\i-1}
          \draw[line width=0.5pt] (c-\i-1)--(c-\ii-1);
          \draw[line width=0.5pt,dashed] (c-\i-1)--(c-\ii-0);
          \fi

          \pgfmathtruncatemacro{\k}{3-\i}
          \foreach \j in {1,...,\k}{
            \path (5+3*\i-1,3+2*\j) node[draw, shape=circle] (c-pt\i-\j) {};

            \ifnum \j=1
            \draw[line width=0.5pt] (c-pt\i-\j)--(c-\i-1);
            \fi

            \ifnum \j>1
            \pgfmathtruncatemacro{\jj}{\j-1}
            \draw[line width=0.5pt] (c-pt\i-\j)--(c-pt\i-\jj);
            \fi

          }

        }

    \end{scope}

    \path (0,0) node[draw, shape=circle] (r) {};
    \draw (r) node[left=0.25cm] {$x$};

    \draw[line width=0.5pt] (r)--(a-1-0);
    \draw[line width=0.5pt] (r)--(a-1-1);

    \draw[line width=0.5pt] (r)--(b-1-0);
    \draw[line width=0.5pt] (r)--(b-1-1);

    \draw[line width=0.5pt] (r)--(c-1-0);
    \draw[line width=0.5pt] (r)--(c-1-1);

    %draw central path
    \foreach \i in {1,2,3}{
      \path (2*\i,1.5*\i) node[draw, shape=circle] (p-\i) {};
      \ifnum \i=1
      \draw[line width=0.5pt] (r)--(p-1);
      \fi
      \ifnum \i>1
      \pgfmathtruncatemacro{\ii}{\i-1}
      \draw[line width=0.5pt] (p-\ii)--(p-\i);
      \fi
    }
  \end{tikzpicture}}\qquad
  \subfigure[The graph $O_{8,3}$.]{
\begin{tikzpicture}[join=bevel,inner sep=0.6mm,line width=0.8pt, scale=0.3]
    \begin{scope}[rotate=90]

        \foreach \i in {1,...,4}{

          \ifnum \i<4
          \path (5+3*\i,0) node[draw, shape=circle] (a-\i-0) {};

          \pgfmathtruncatemacro{\k}{4-\i}
          \foreach \j in {1,...,\k}{
            \path (5+3*\i,-2*\j) node[draw, shape=circle] (a-p\i-\j) {};

            \ifnum \j=1
            \draw[line width=0.5pt] (a-p\i-\j)--(a-\i-0);
            \fi

            \ifnum \j>1
            \pgfmathtruncatemacro{\jj}{\j-1}
            \draw[line width=0.5pt] (a-p\i-\j)--(a-p\i-\jj);
            \fi
          }
          \fi

          \ifnum \i=4
          \path (5+3*\i,0) node[draw, shape=circle,fill] (a-\i-0) {};
          \fi

          \ifnum \i>1
          \pgfmathtruncatemacro{\ii}{\i-1}
          \draw[line width=0.5pt] (a-\i-0)--(a-\ii-0);
          \fi

        }

        \foreach \i in {1,2,3}{

          \path (5+3*\i-1,3) node[draw, shape=circle] (a-\i-1) {};

          \ifnum \i<3
          \draw[line width=0.5pt,dashed] (a-\i-1)--(a-\i-0);
          \fi
          \ifnum \i=3
          \draw[line width=0.5pt] (a-\i-1)--(a-\i-0);
          \fi

          \ifnum \i>1
          \pgfmathtruncatemacro{\ii}{\i-1}
          \draw[line width=0.5pt] (a-\i-1)--(a-\ii-1);
          \draw[line width=0.5pt,dashed] (a-\i-1)--(a-\ii-0);
          \fi

          \pgfmathtruncatemacro{\k}{4-\i}
          \foreach \j in {1,...,\k}{
            \path (5+3*\i-1,3+2*\j) node[draw, shape=circle] (a-pt\i-\j) {};

            \ifnum \j=1
            \draw[line width=0.5pt] (a-pt\i-\j)--(a-\i-1);
            \fi

            \ifnum \j>1
            \pgfmathtruncatemacro{\jj}{\j-1}
            \draw[line width=0.5pt] (a-pt\i-\j)--(a-pt\i-\jj);
            \fi
          }

        }

    \end{scope}

    \begin{scope}[rotate=210]

        \foreach \i in {1,...,4}{

          \ifnum \i<4
          \path (5+3*\i,0) node[draw, shape=circle] (b-\i-0) {};

          \pgfmathtruncatemacro{\k}{4-\i}
          \foreach \j in {1,...,\k}{
            \path (5+3*\i,-2*\j) node[draw, shape=circle] (b-p\i-\j) {};

            \ifnum \j=1
            \draw[line width=0.5pt] (b-p\i-\j)--(b-\i-0);
            \fi

            \ifnum \j>1
            \pgfmathtruncatemacro{\jj}{\j-1}
            \draw[line width=0.5pt] (b-p\i-\j)--(b-p\i-\jj);
            \fi
          }
          \fi

          \ifnum \i=4
          \path (5+3*\i,0) node[draw, shape=circle,fill] (b-\i-0) {};
          \fi

          \ifnum \i>1
          \pgfmathtruncatemacro{\ii}{\i-1}
          \draw[line width=0.5pt] (b-\i-0)--(b-\ii-0);
          \fi

        }

        \foreach \i in {1,2,3}{

          \path (5+3*\i-1,3) node[draw, shape=circle] (b-\i-1) {};

          \ifnum \i<3
          \draw[line width=0.5pt,dashed] (b-\i-1)--(b-\i-0);
          \fi
          \ifnum \i=3
          \draw[line width=0.5pt] (b-\i-1)--(b-\i-0);
          \fi

          \ifnum \i>1
          \pgfmathtruncatemacro{\ii}{\i-1}
          \draw[line width=0.5pt] (b-\i-1)--(b-\ii-1);
          \draw[line width=0.5pt,dashed] (b-\i-1)--(b-\ii-0);
          \fi

          \pgfmathtruncatemacro{\k}{4-\i}
          \foreach \j in {1,...,\k}{
            \path (5+3*\i-1,3+2*\j) node[draw, shape=circle] (b-pt\i-\j) {};

            \ifnum \j=1
            \draw[line width=0.5pt] (b-pt\i-\j)--(b-\i-1);
            \fi

            \ifnum \j>1
            \pgfmathtruncatemacro{\jj}{\j-1}
            \draw[line width=0.5pt] (b-pt\i-\j)--(b-pt\i-\jj);
            \fi
          }

        }

    \end{scope}

    \begin{scope}[rotate=330]

        \foreach \i in {1,...,4}{

          \ifnum \i<4
          \path (5+3*\i,0) node[draw, shape=circle] (c-\i-0) {};

          \pgfmathtruncatemacro{\k}{4-\i}
          \foreach \j in {1,...,\k}{
            \path (5+3*\i,-2*\j) node[draw, shape=circle] (c-p\i-\j) {};

            \ifnum \j=1
            \draw[line width=0.5pt] (c-p\i-\j)--(c-\i-0);
            \fi

            \ifnum \j>1
            \pgfmathtruncatemacro{\jj}{\j-1}
            \draw[line width=0.5pt] (c-p\i-\j)--(c-p\i-\jj);
            \fi
          }
          \fi

          \ifnum \i=4
          \path (5+3*\i,0) node[draw, shape=circle,fill] (c-\i-0) {};
          \fi

          \ifnum \i>1
          \pgfmathtruncatemacro{\ii}{\i-1}
          \draw[line width=0.5pt] (c-\i-0)--(c-\ii-0);
          \fi

        }

        \foreach \i in {1,2,3}{

          \path (5+3*\i-1,3) node[draw, shape=circle] (c-\i-1) {};

          \ifnum \i<3
          \draw[line width=0.5pt,dashed] (c-\i-1)--(c-\i-0);
          \fi
          \ifnum \i=3
          \draw[line width=0.5pt] (c-\i-1)--(c-\i-0);
          \fi

          \ifnum \i>1
          \pgfmathtruncatemacro{\ii}{\i-1}
          \draw[line width=0.5pt] (c-\i-1)--(c-\ii-1);
          \draw[line width=0.5pt,dashed] (c-\i-1)--(c-\ii-0);
          \fi

          \pgfmathtruncatemacro{\k}{4-\i}
          \foreach \j in {1,...,\k}{
            \path (5+3*\i-1,3+2*\j) node[draw, shape=circle] (c-pt\i-\j) {};

            \ifnum \j=1
            \draw[line width=0.5pt] (c-pt\i-\j)--(c-\i-1);
            \fi

            \ifnum \j>1
            \pgfmathtruncatemacro{\jj}{\j-1}
            \draw[line width=0.5pt] (c-pt\i-\j)--(c-pt\i-\jj);
            \fi

          }

        }

    \end{scope}

    \path (0,0) node[draw, shape=circle] (r) {};
    \draw (r) node[left=0.25cm] {$x$};

    \draw[line width=0.5pt] (r)--(a-1-0);
    \draw[line width=0.5pt] (r)--(a-1-1);

    \draw[line width=0.5pt] (r)--(b-1-0);
    \draw[line width=0.5pt] (r)--(b-1-1);

    \draw[line width=0.5pt] (r)--(c-1-0);
    \draw[line width=0.5pt] (r)--(c-1-1);

    %draw central path
    \foreach \i in {1,2,3,4}{
      \path (2*\i,1.5*\i) node[draw, shape=circle] (p-\i) {};
      \ifnum \i=1
      \draw[line width=0.5pt] (r)--(p-1);
      \fi
      \ifnum \i>1
      \pgfmathtruncatemacro{\ii}{\i-1}
      \draw[line width=0.5pt] (p-\ii)--(p-\i);
      \fi
    }
  \end{tikzpicture}}
}

  \caption{The graphs $O_{7,3}$ and $O_{8,3}$. Dashed edges are optional. Black vertices form an optimal resolving set.}
  \label{fig:extr-outerplanar}
\end{figure}
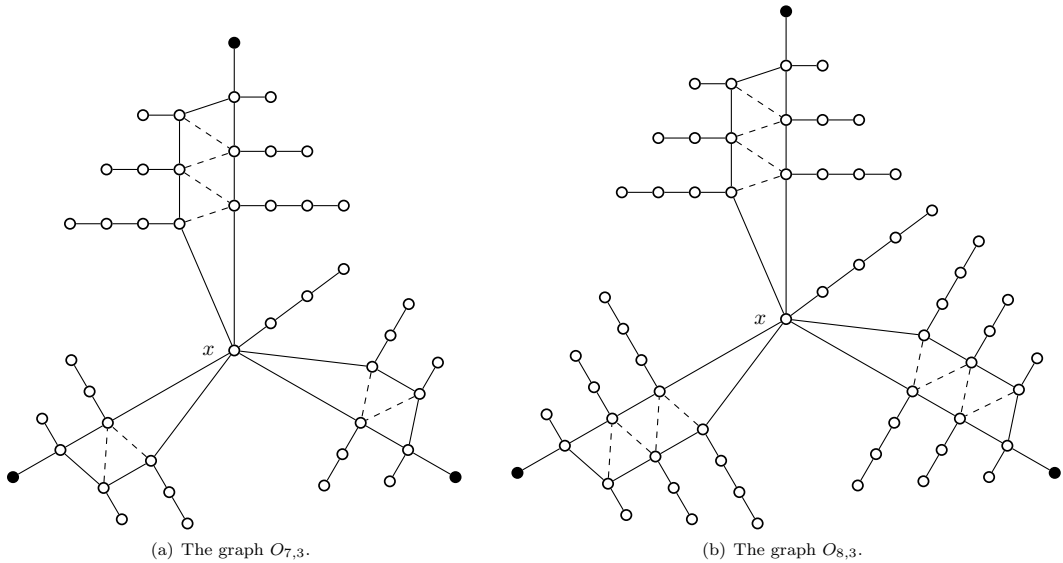

\section{Conclusion}\label{sec:conclu}

For trees and outerplanar graphs, we know that $n=\mathcal{O}(kd^2)$
and this is tight. We do not know whether our other bounds are
tight. It would be interesting to further study the classes of graphs
of fixed treewidth~$w$ (or the more restricted case of $w$-trees) and
the class of chordal graphs. We have proved that
$n=\mathcal{O}(kd^{3w+3})$ for constant $w$
(Corollary~\ref{coro:bounded-treewidth}) and $n=\mathcal{O}(f(k)d^2)$
for chordal graphs, where $f$ is doubly-exponential
(Corollary~\ref{coro:chordal}). Can these bounds be improved?
Moreover, Corollary~\ref{cor:minor-rankwidth-bounds} gives a bound in
terms of rankwidth. Trying to get a similar result in terms of
cliquewidth seems to be a natural follow-up.

Another interesting problem is to determine the best possible bound
for planar graphs, that is, whether our $n=\mathcal{O}(d^4k^4)$ bound
that follows from Theorem~\ref{thm:minorfree} can be improved. Note
that $n=\mathcal{O}(d^2)$ holds when the metric dimension is~$2$,
indeed in this case we have $n \leq d^2+2$ for any
graph~\cite{CEJO00,KRR96}. This quadratic bound is matched by any
square grid, which has metric dimension~$2$ and $n=d^2$. Nevertheless,
there are planar graphs with metric dimension $3$ and order
$\Theta(d^3)$. Such a family of graphs can be described as
follows. Pick any integer $t$ and consider $t$ disjoint copies $G_1,
G_2, \ldots G_t$ of a $t \times t$ grid. For $i$ between $1$ and
$t-1$, add an edge between the top left corners of $G_i$ and $G_{i+1}$
and another edge between the top right corners of $G_i$ and
$G_{i+1}$. The diameter of this graph is $4t$ and its order is
$t^3$. Moreover the top corners of $G_1$ together with the bottom left
corner of $G_t$ form a resolving set of size $3$. We do not know
whether there are planar graphs with small metric dimension and order
$\Theta(d^4)$.

For the smaller class of treewidth~$2$ graphs (that is, $K_4$-minor
free graphs), we know that $n=\mathcal{O}(d^3k^3)$
(Theorem~\ref{thm:minorfree}) and $n=\mathcal{O}(d^9k)$
(Corollary~\ref{coro:bounded-treewidth}), but we doubt that these
bounds are optimal. We remark that our proof method for outerplanar
graphs does not seem to be easily generalizable to this class.

\section*{Acknowledgements}

The authors are grateful to Julien Cassaigne who found the family of
planar graphs described in the conclusion, and to the anonymous
referees for their careful reading and constructive remarks.

%% \section{Questions (for our private usage)}

%% \begin{question}
%% Better bound for planar graphs? We think that if $d$ is the diameter and $k$ is the metric dimension, there should be at most $\mathcal{O}(d^2k)$ vertices in the graph.
%% A first thing to prove would be two prove that a planar graph with metric dimension 3 has order $\mathcal{O}(d^2)$.
%% \end{question}

%% \begin{question}
%% Do we have a converse of Proposition~\ref{md-vcdim} to complete the dichotomy ?
%% \end{question}

%% \begin{question}
%% Can we get rid of the diameter in Proposition~\ref{prop:md-vcdim-2} and have an expression only using the VC dimension of the hypergraph $\cH_\ell (G)$ (defined with the balls of radius $\ell$) ? Do we have example of graphs with different $\dVC$ and $\dVC^*$ ?
%% \end{question}

%% \begin{question}
%% If a class of graphs is stable by powers and has bounded VC dimension, does it have bounded distance-VC dimension ? It seems to be true for interval graphs (to prove properly).
%% \end{question}

%% \begin{question}
%% Relation with the rank width ? (the dual vc dimension is bounded, but can we have a better bound ?)
%% \end{question}

\end{document}